\documentclass[12pt, reqno, a4paper]{amsart}
\usepackage{amsmath,mathrsfs}
\usepackage{amssymb}
\usepackage{calligra}
\usepackage{dsfont}

\newcommand\NoBlackBoxes{\global\overfullrule0pt}
\NoBlackBoxes

\usepackage{natbib}

\newcommand{\N}{\mathbb{N}}

\makeatletter
\let\serieslogo@\relax
\let\@setcopyright\relax

\newtheorem{definition}{Definition}[section]
\newtheorem{theorem}[definition]{Theorem}
\newtheorem{proposition}[definition]{Proposition}
\newtheorem{lemma}[definition]{Lemma}

\newenvironment{remark}[1][Remark]{\begin{trivlist}
\item[\hskip \labelsep {\bfseries #1}]}{\end{trivlist}}

\renewcommand{\O}{\Omega}

\newcommand{\ind}{\operatorname{1}}

\numberwithin{equation}{section}

\begin{document}

\setcounter{page}{1}

\title[Equi-Energy sampling does not converge rapidly]{Equi-Energy sampling does not converge rapidly on the mean-field Potts model with three colors close to the critical temperature}

\author[Mirko Ebbers]{Mirko Ebbers}
\address[Mirko Ebbers]{Fachbereich Mathematik und Informatik,
Universit\"at M\"unster,
Einsteinstra\ss e 62,
48149 M\"unster,
Germany}

\email[Mirko Ebbers]{mirkoebbers@uni-muenster.de}

\author[Matthias L\"owe]{Matthias L\"owe}
\address[Matthias L\"owe]{Fachbereich Mathematik und Informatik,
Universit\"at M\"unster,
Einsteinstra\ss e 62,
48149 M\"unster,
Germany}

\email[Matthias L\"owe]{maloewe@uni-muenster.de}

\begin{abstract}
Equi-Energy Sampling (EES, for short) is a method to speed up the convergence of the Metropolis chain, when the latter is slow. We show that there are still models like the mean-field Potts model, where EES does not converge rapidly in certain temperature regimes. Indeed we will show that EES is slowly mixing on the mean-field Potts model, in a regime below the critical temperature. Though we will concentrate on the Potts model with three colors, our arguments remain valid for any number of colors $q \ge 3$, if we adapt the temperature regime. For the situation of the mean-field Potts model this answers a question posed in \cite{HuaKou}.
\end{abstract}

\thanks{Research of the second author was
funded by the Deutsche Forschungsgemeinschaft (DFG, German Research Foundation) under Germany 's Excellence Strategy
EXC 2044 –390685587, Mathematics M\"unster: Dynamics –Geometry -Structure}

\date{\today}

\subjclass[2010]{60J10,60K35, 82B05}

\keywords{Equi-Energy Algorithm, Potts model, Swapping Algorithm, Metropolis Algorithm, Markov Chain Monte Carlo methods, Curie-Weiss model}

\newcommand{\wlim}{\mathop{\hbox{\rm w-lim}}}
\newcommand{\na}{{\mathbb N}}
\newcommand{\re}{{\mathbb R}}
\renewcommand{\epsilon}{\varepsilon}

\newcommand{\vep}{\varepsilon}

\maketitle

\section{Introduction}
Sampling methods are of utmost importance in applied mathematics, e.g. in Bayesian statistics, computational physics, econometrics, or computational biology.
In many cases one wants to sample a random element drawn from a finite set $\Omega$ according to a probability distribution $\pi$ on $(\O, \mathcal P(\O))$. But even this problem may be less trivial than it sounds. Sometimes $\Omega$ may be finite, yet very large. E.g. when modeling a ferromagnetic material on $N$ atoms, $\O$ is of the form
$\{-1,+1\}^N$ and for real size systems, $N$ is of the order $10^{23}$, thus $|\O|$ is of the order $2^{{10}^{23}}$. Hence a straight-forward Monte-Carlo simulation would take exponentially long in the system size $N$ and thus would be much too expensive. In other situations, the size of $\O$ may not even be known, as e.g.\ in the so-called knapsack problem (\cite{LM01}).

One potential solution of this problem lies in the use of
Markov Chain Monte Carlo (MCMC, for short) algorithms.
They rely on an aperiodic and irreducible Markov chain on $\O$ that has $\pi$ as its invariant (i.e.\ stationary) distribution. One runs this Markov chain, stops it after some long enough time, and takes the current state as a sample element. The ergodic Theorem for Markov chains ensures that this element is almost distributed according to $\pi$, given one has waited long enough.
This method immediately raises two questions:
\begin{enumerate}
\item Can one find for each $\pi$ a Markov chain that converges in distribution to $\pi$? I.e.\ can one find for each $\pi$ an irreducible and aperiodic Markov chain that has $\pi$ as its invariant measure? This question is answered in the affirmative by the Metropolis-Hastings chain (see e.g. \cite{haeggstroem}).
\item How long do we need to wait to get a sample with a distribution that is reasonably close to $\pi$? If this waiting time is polynomial in the problem instance we speak about fast or rapid convergence, otherwise, in particular, if the mixing time is exponential in the problem instance, we will say the algorithm converges torpidly or slowly.
\end{enumerate}
It is, however, well known that
like the Glauber dynamics the Metropolis-Hastings algorithm usually converges slowly, when the target distribution is multi-modal.
Such situations occur e.g. in statistical physics in the presence of a phase transition. Hence slow convergence applies
to a number of interesting situations, among them the low temperature phase of the Curie-Weiss model (see e.g. the discussion in \cite{MosselSly}).
In the next section we will introduce a close relative of the Curie-Weiss model, the three state mean-field Potts model. This will be our test model for the EES to be introduced in Section 3. In Section 4 we will show that this sampler mixes slowly when applied to the (three state) mean-field Potts model in a certain temperature regime. Here, a key argument is based on a property of the mean-field Potts model that is closely related to its first order phase transition (see Lemma \ref{lemma loc min} and the remark below it): The limit point of the order parameter $m_N$ (cf. \eqref{mN}) at high temperatures remains a local maximum of its distribution also in a certain part of the low temperature regime. Hence a Metropolis-Hastings chain started in a neighborhood of this high temperature limit point will typically not escape from this neighborhood in polynomial time (in the part of the low temperature regime described above). But then also EES cannot improve the performance of the Metropolis-Hastings algorithm, because there are simply no observations of the global maxima of the distribution of $m_N$ the algorithm could jump at. The proof is completed by combining this observation with the very powerful technical argument of conductance, also known as Cheeger's inequality (see Theorem \ref{conductance}).


\section{The mean-field Potts model}
Let us now introduce the mean-field Potts model. Consider the space $\Omega=E^N$, where $E=\{1,2, \ldots,q\}$, $ q \in \N$, $q \ge 3$, and $N\in \na$ (to avoid some complications in the future, we can think of $N$ being a multiple of $q$). The elements of $E$ are sometimes referred to as colors.

For convenience, in this note, we will restrict to the case of $q =3$, i.e.\ the mean-field Potts model with three states or colors taken from the set
$E=\{1,2,3\}$. This from now on will be our standing assumption for the rest of the note. However, we remark that our argument remains valid for general $q \ge 3$, if one changes the regime of temperatures appropriately. We will come back to this remark later.

On $\O$ we construct an energy function given by
$$
H_N(\sigma):= \frac 1{2N} \sum_{i,j=1}^N \delta_{\sigma_i=\sigma_j},\quad \sigma \in \Omega.
$$
Here $\delta_A$ denotes the indicator function for an event $A$ (which is formally the Dirac measure for the event $A$, to stress that our notation is consistent with our later use of $\delta$).
Note that $H_N$ can be written as a function of the vector
\begin{multline}\label{mN}
m_N(\sigma):=(m^1_N(\sigma),m_N^2(\sigma),m_N^3(\sigma))\\:= \left(\frac 1N \sum_{i=1}^N \delta_{\sigma_i=1}, \frac 1N \sum_{i=1}^N \delta_{\sigma_i=2},\frac 1N \sum_{i=1}^N \delta_{\sigma_i=3}\right).
\end{multline}
Indeed, one easily checks that
$$
H_N(\sigma)= \frac N{2} \sum_{c=1}^3 (m_N^c(\sigma))^2.
$$
$m_N$ is therefore called an order parameter of the model.
With $H_N$ we associate a Gibbs measure $\pi_\beta$ at inverse temperature $\beta >0$, i.e.
$$
\pi_\beta(\sigma)= \frac{e^{\beta H(\sigma)}}{Z_\beta}, \quad \sigma \in \O.
$$
Here
$Z_\beta = \sum_\tau e^{\beta H(\tau)}$ is the partition function. Note that conventionally in statistical physics the energy function $H$ would carry an additional minus sign, and the Gibbs measure (as well as the partition function) would be defined in terms of the exponential of minus $\beta$ times that energy. Since the two minus signs would cancel and lead to the same definition of the Gibbs measure our Gibbs measure does not carry these conventional minus signs.

The mean-field Potts model was studied in a variety of papers. We refer to \cite{EllisWang} and \cite{Kestenschonman}, who showed that there is a critical inverse temperature $\beta_c$. This critical inverse temperature in the 3 states mean-field Potts model equals $\beta_c= 4\log 2$ (cf. \cite{Cuffetal}, which discusses the very interesting phenomenon of a temperature-dependent cut-off effect for the Glauber dynamics of the model).

At $\beta_c$ the model undergoes a first order phase transition. More precisely, the order parameter $m_N$ of the model converges in distribution to $\mathfrak{a}_0:=(\frac 13, \frac 13, \frac 13)$, when $\beta<\beta_c$.
At smaller temperatures one observes the following:
For $\beta \ge \beta_c$ there is $1 > m^*(\beta)> \frac 13$ and vectors $\mathfrak{a}_1(\beta), \mathfrak{a}_2(\beta), \mathfrak{a}_3(\beta) \in \re^3$ such that the vector $\mathfrak{a}_i(\beta)$ has $m^*(\beta)$ in its i'th component and all other components are equal such that they sum up to one.
For $\beta>\beta_c$ the distribution of $m_N$ converges to $\frac 13 \sum_{i=1}^3 \delta_{\mathfrak{a}_i(\beta)}$. Here $\delta$ denotes the Dirac-measure.
Finally, for $\beta=\beta_c$ there are, moreover, weights $\lambda_1, \lambda_2 >0$ that sum up to 1, such that the distribution of $m_N$ converges to $\lambda_1 \delta_{\mathfrak{a}_0}+ \lambda_2 \sum_{i=1}^3 \delta_{\mathfrak{a}_i(\beta_c)}$.

The phase transition is of first order, since $m^*(\beta_c) > \frac 13$, i.e. the jump is discontinuous. Moreover, the vector $\mathfrak{a}_0$ remains a local
maximum of the distribution of $m_N$ for some temperatures below the critical temperature.
Such a behavior can also be observed for general values of $q \ge 3$ at other values for $\beta_c$.
We will come back to this fact in Section 4, Lemmas \ref{lem1} and \ref{lemma loc min}, because it is of utmost importance for the proof of Theorem \ref{Centraltheorem} to be given in Section 4. 

\section{Equi-Energy Sampling}
Various modifications of the Metropolis-Hastings algorithm have been proposed to speed up its convergence. Among them the so-called swapping algorithm (see \cite{GeyerMCMCmaximumLikelihood}), the exchange Monte Carlo method (see \cite{exMC}), parallel tempering (see \cite{orlandini}) and the simulated tempering algorithm (see \cite{marinari_parisi}, \cite{GeyerThompsonAMCMC}, and \cite{madras}) are very popular in applications.
Another variant are Multicanonical Monte Carlo Simulations, introduced by \cite{BN92}, also see \cite{B00}. It is related to umbrella sampling (see \cite{umbrella}) and is
close in spirit to the swapping algorithm, simulated tempering, as well as EES. A major difference, however, is, how an a priori estimate of the probability distribution of interest is produced. Therefore, we have not yet been able to analyze so far, whether Multicanonical Monte Carlo Simulations suffer from the same shortcomings as swapping, parallel tempering or EES (see next paragraph).

In many situations the algorithms described in the previous paragraph seem indeed to be able to improve the convergence of the Metropolis chain, however, there are also negative theoretical results about these
algorithms. \cite{MadrasZhengCW} show that the swapping chain converges quickly for the Curie--Weiss model. On the other hand,  \cite{BhatnagarRandallTorpidMixingPotts} and \cite{BR16} prove that both, the swapping algorithm and simulated tempering, are slowly mixing for the 3-state Potts model and conjecture that this is caused by the first order phase transition in the Potts model (also see our discussion in the remark following Lemma \ref{lemma loc min}). Qualitative properties of the swapping algorithm and parallel tempering were studied in \cite{DDN18}.
A first rapid convergence result for the Swapping Algorithm in an disordered situation was proved in \cite{loewe_vermet_swap}. \cite{EbbersLoweREM} show that in disordered models the conjecture by Bhatnagar and Randall is not correct. They prove that the Swapping Algorithm mixes slowly on the Random Energy Model, even though this model has only a third order phase transition.
In the Blume-Emery-Griffiths model both, rapid or torpid mixing may occur as was shown in \cite{EKLV}.

Another idea to improve the performance of the Metropolis chain is the so called Equi-energy sampling algorithm (see e.g. \cite{KZW}). This model was tested on the Ising model in \cite{HuaKou} and the question, how fast it converges, was posed. For the Potts model, we will answer this question in the next section. In particular, we will show that EES may be slowly mixing on relevant models from statistical mechanics. Variants of EES were studied, among others, in \cite{Baragatti_etal}.

The principle observation to motivate EES is that a main obstacle to fast mixing is the presence of a phase transition in the model. This, in turn, may be characterized by a multi-modal distribution of a macroscopic observable. Usually, then the (projected) Metropolis chain enters one of the modes rapidly and stays there for an exponentially long time. The EES tries to avoid this behavior by introducing shortcuts in the state space. These shortcuts are created by the observations of Metropolis chains at higher temperatures where the above mentioned modes are less pronounced or possibly not even present. More precisely, additionally to the Metropolis steps one allows also for jumps to points of the same or a similar energy as the present one, given one has observed these points already at higher temperatures (otherwise, the algorithm would require the exact structure of the energy function, in which case simulations would probably be pointless). The EES has been discussed in \cite{KZW}, its convergence was shown in the same article, and, using a different technique, in \cite{AJDD}. We will now give an exact description of a version of this algorithm.

Let us first briefly recall the Metropolis-Hastings algorithm, which is the basis of the EES.
To define the first let $K_{gen}$ denote the following aperiodic, symmetric and irreducible Markov chain on $\Omega$:
\begin{equation} \label{Definition_Kgen}
K_{gen}(\sigma,\tau)=\left\{\begin{array}{ll}
\frac 12 & \mbox{if } \sigma=\tau\\
\frac 1 {4N} & \mbox{if } d_H(\sigma, \tau)=1 \\
0 & \mbox{otherwise.}
\end{array}
\right.
\end{equation}
Here $d_H$ is the Hamming distance and $K_{gen}$ is a Markov chain, because every $\sigma \in \O$ has $2N$ neighbors. Define the corresponding Metropolis-Hastings chain for the probability $\pi_\beta$
as $T_\beta(\cdot, \cdot)$:
\begin{equation} \label{Definition_Metropolis}
T_\beta(x,y)=\left\{\begin{array}{ll}
K_{gen}(x,y) & \mbox{if } x\neq y \mbox{ and } H(y) \ge H(x) \\
K_{gen}(x,y) \frac{\pi_\beta(y)}{\pi_\beta(x)} & \mbox{if } x\neq y \mbox{ and } H(y) < H(x) \\
1-\sum_{z\neq x} T_\beta(x,z) & \mbox{otherwise.}
\end{array}
\right.
\end{equation}
Note that $T_\beta(\cdot, \cdot)$ sometimes is slow in natural situations, e.g. when sampling from the low temperature distribution of the Curie-Weiss model (see e.g. \cite{MadrasPiccioni}, of course, the $T_\beta$ has to be adapted to the situation of the Curie-Weiss model).
To speed up its convergence, we consider the EES. To define it, we first introduce a sequence of energy levels:
$$
\inf_x H(x):= h_0 < h_1 < \ldots < h_M= \sup H(x)
$$
In our context, it is easily checked that
\begin{equation}\label{energy}
\inf_x H(x)=\frac N6 \quad \text{and} \quad \sup_x H(x)=\frac N2
\end{equation}
which will be used later.
Moreover, introduce a sequence of inverse temperature levels
$$
0=\beta_0 < \beta_1 < \ldots < \beta_M= \beta
$$
where we assume that $\beta$ is the temperature we want to sample from. It will often be convenient to take $\beta_i= i\frac{\beta}M$. Note that $M$ may and will depend on $N$, which is not made explicit in \cite{KZW}, otherwise our construction, so far, agrees with the construction in \cite{KZW}. We will make an explicit choice for $M$ and give reasons for this choice, after the description of the algorithm

For this, we will also need a dummy state $\iota$ and define
$\tilde \O := \O \cup \{\iota\}.$
Let $\mathcal M$ be an $(M+1) \times |\O|$ matrix over $\tilde \O$, which is initially filled with $\iota$, only.

The EES consists of alternations between two steps. One is a usual Metropolis step at a (random) temperature level $\beta_i$. The other one is an equi-energy jump at the same temperature, if $i\ge 1$. At inverse temperature 0 there are only Metropolis moves. We store the resulting energies of the states we see at temperature $\beta_i$ by entering them into the $i$'th row of the matrix $\mathcal M$, if the state has not been seen before. In this case it replaces one of the $\iota$'s (in a pre-described order). To explain the equi-energy step assume that the chain is at temperature $\beta_i, i\ge 1$ and in state $\sigma$. We determine the energy level $k$, such that $h_{k-1} < H(\sigma) \le h_k$ and choose (with equal probabilities) a state $\tau$ from all states $\tau'$
with $h_{k-1} < H(\tau') \le h_k$, which we have already seen at temperature level $\beta_{i-1}$. This new state is accepted with probability $\min\{1, \frac{\pi_{\beta_i}(\tau) \pi_{\beta_{i-1}}(\sigma)}{ \pi_{\beta_i}(\sigma) \pi_{\beta_{i-1}}(\tau)}\}$. Otherwise, especially, if we have not seen any state in the same energy band in the $i-1$st row of $\mathcal M$, the chain stays where it is. We denote the corresponding transition matrix (on $\O$) by $Q_i$. Note that $Q_i$ in general depends on time. We will not make this explicit, because we will just analyze the algorithm in the ''best case scenario'', where the matrix $\mathcal M$ does not contain any $\iota$'s anymore. However, under this assumption, we will still be able to show that EES is slowly converging on the three state mean-field Potts model in a certain temperature regime. Formally
\begin{equation}\label{Qi}
Q_i(\sigma, \tau):=Q_{n,i}(\sigma, \tau):= \frac{1}{|B_{n,k}|}\min\{1, \frac{\pi_{\beta_i}(\tau) \pi_{\beta_{i-1}}(\sigma)}{ \pi_{\beta_i}(\sigma) \pi_{\beta_{i-1}}(\tau)}\}\ind_{\tau \in B_{n,k}}.
\end{equation}
Here $n$ is the time variable and $B_{n,k}$ is the set of states $\tau'$
with $h_{k-1} < H(\tau') \le h_k$, which we have already seen at temperature level $\beta_{i-1}$ by time $n$.

One might expect, that we indeed use all states we have seen previously, rather than the ones we explored with the chain at temperature $\beta_{i-1}$. However, there is hardly any difference between the two chains, because if temperatures are very different the chains will typically also see states of very different energies. Our choice has the advantage that it is easy to see that the global chain to be described below is reversible and moreover, it agrees with the choice in the literature, see \cite{KZW}.

Based on this, we build a matrix that describes the movement of all particles simultaneously. This operator $\mathcal{R}$ will be a matrix on $\O^{M+1}$, of course. We lift the movement of the i'th particle to $\O^{M+1}$ by building
$$\mathcal{Q}_i:= \bigotimes_{j=0}^{i-1} I \otimes Q_i \bigotimes_{j=i+1}^M I$$
where $I$ is the identity matrix. Similarly, we consider the matrix $\mathcal{T}_i$ that lifts the Metropolis step $T_{\beta_i}$ to $\O^{M+1}$, i.e. we consider
$$\mathcal{T}_i:= \bigotimes_{j=0}^{i-1} I \otimes T_{\beta_i} \bigotimes_{j=i+1}^M I.$$
Combining these operators the EES is defined by
$$\mathcal{R}=\frac{1}{(M + 1)^3} \sum_{j,k,l=0}^M \mathcal{Q}_j\mathcal{T}_k\mathcal{Q}_l.$$
Note that the versions of the EES given in \cite{KZW} and \cite{AJDD} differ from each other and also our version is slightly different from those. However, the spirit of the algorithms is the same.

In the sequel, we will only consider a number of energy levels $M$ that depends linearly on $N$, such that $M=dN$. We will furthermore assume that $h_i$ are equi-distant. Indeed, this choice of $M$ is somewhat arbitrary, allowing for a polynomial dependence between $M$ and $N$ would not alter the algorithm much. However,
choosing $M$, e.g. exponentially large in $N$, would lead to empty, or almost empty energy
bands which would make the equi-energy step obsolete. Moreover, it would obviously lead to exponential relaxation times (in $N$), because exponentially many temperatures have to be simulated. On
the other hand, having $M$ too small, e.g.\ constant, leads to almost non-interactive components (i.e.\ an equi-energy jump is almost never accepted)
and EES stands no chance of increasing the
speed of convergence compared to the standard Metropolis algorithm.

Of course, eventually we will only be interested in the $M+1$'st coordinate of this Markov chain. However, studying it entirely, seems easier. First of all, let us note that indeed, the distribution of the $M+1$'st coordinate converges to $\pi_\beta$.
\begin{theorem}\label{convergence result}
The distribution of the $M+1$'st coordinate converges to $\pi_\beta$ as time tends to infinity.
\end{theorem}
\begin{proof}
This is the content of \cite{AJDD} for their version of the EES. For our version the assertion follows from the ergodic theorem for Markov chains. Indeed, denote by $\mathcal{S}$ the Markov chain on $\O^{M+1}\times \mathfrak{M}$, where $\mathfrak{M}$ is the space of all $(M+1) \times 3^N$ matrices. $\mathcal{S}$ will behave in its first component like $\mathcal{R}$ while in the second component we keep record of the filling of $\mathcal{M}$.
Observe that each $T_i:=T_{\beta_i}$ is reversible with respect to $\pi_{\beta_i}$ and $\mathcal{M}$ does not play any role for it. On the other hand, once we reach a situation where $\mathcal{M}$ is entirely filled with states different from $\iota$ (we denote this state of $\mathcal{M}$ by $M_0$ in the second coordinate of $\mathcal{S})$, i.e.\ we have seen all states at all temperatures, also all the equi-energy steps $Q_i$ are reversible with respect to $\pi_{\beta_i}$. This is, because $Q_i(\sigma, \tau)>0$, if and only if, $\sigma$ and $\tau$ lie in the same energy band and follows from the construction of the transition probabilities.
Thus, once $M_0$ is reached -- which happens almost surely in finite time -- $\mathcal{S}$ is reversible with respect to $$
\pi:=\prod_{i=0}^M \pi_{\beta_i} \times \delta_{M_0}.$$
Then the convergence follows from the convergence theorem for Markov chains. This, in particular, yields the assertion of the theorem.
\end{proof}

The proof is somewhat misleading, as it seems to indicate, that for exponentially large state spaces there is no hope that EES may converge in polynomial time, since first the state $M_0$ has to be reached. However, if we consider the high temperature situation $\beta <\beta_c$ in the Potts model the Metropolis-Hastings chain converges to its invariant distribution in polynomial time, even without any equi-energy steps. 

On the other hand, we will see in the next section that in part of the low temperature regime $\beta > \beta_c$ the situation is even worse. Even, when we start $\mathcal{S}$ in the optimal state $M_0$ in its second component, i.e.\ when we assume the second component is already in $M_0$, the mixing time may be exponential.

\section{Torpid mixing of EES on the low temperature mean-field Potts model}
We now come to the central result of the note.
\begin{theorem}\label{Centraltheorem}EES is slowly mixing on the 3-state mean-field Potts model, when $\beta_c<\beta< 3$, even when the second component of the Markov chain $\mathcal{S}$ introduced in the proof of Theorem \ref{convergence result} above is in state $M_0$.
\end{theorem}
We will prepare the proof of the theorem by explaining the ideas and stating some lemmas.
In the proof of the theorem we will exploit one of the main differences between the mean-field Potts model and the Curie-Weiss model (i.e.\ when $E=\{1,2\}$) at low temperatures.
This difference lies in the fact, that in the Curie-Weiss model the state where both colors occur equally often is a local minimum of the Gibbs measure at all low temperatures, while it is a local maximum of the Gibbs measure in the mean-field Potts model for some temperatures in the low temperature regime (also see Lemma \ref{lemma loc min}). In particular, in the Curie-Weiss model, the Gibbs measure is flat in this state at the critical temperature while it exposes a local maximum in this state at the critical temperature in the Potts model. Thus, in the latter, EES will be very reluctant to move far away from a state $\sigma$ with $m_N(\sigma) \approx (\frac 13, \frac 13, \frac 13)$.
This is the core idea, even though the technical steps are somewhat more involved.

Let $c_1, c_2, c_3$ be numbers in $[0,1]$ that add up to 1 and such that $c_i N$ is an integer for each $i=1,2,3$. Then for $\sigma_c$ such that $m_N(\sigma_c)=(c_1, c_2, c_3)=:c$ we have that
\begin{eqnarray*}
\pi_\beta(m_N(\sigma)=c)&=& \binom{N}{Nc_1, Nc_2, Nc_3} \frac{\exp(\beta H_N(\sigma_c))}{Z_\beta}\nonumber\\
\end{eqnarray*}
\begin{eqnarray}\label{stirling}
&=& \frac 1N \exp(-N(\sum_{i=1}^N c_i \log c_i +\Delta(c)))\frac{\exp(\beta H_N(\sigma_c))}{Z_\beta}\nonumber\\
&=& \frac{\exp(N(f(c)+\Delta(c)))}{N Z_\beta}
\end{eqnarray}
where
$$f(c)= \sum_{i=1}^3 \left(\frac \beta 2 c_i^2-c_i \log c_i\right)$$
and $\Delta (c)$ is $o(1)$. Note that we used Stirling's formula to derive the second equality in \eqref{stirling} and the fact that we can rewrite $H_N(\sigma)$ as $H_N(\sigma)= \frac 12 \sum_{i=1}^3 c_i^2$, if $m_N(\sigma)=c$.

Letting $\mathcal{C}:= \{(c_1,c_2,c_3) \in [0,1]^3: \sum_{i=1}^3 c_i=1\}$ to be the domain of $f$ (and the set of all probabilities on the space $E$), Gore and Jerrum show:
\begin{lemma}(cf.\cite[Proposition 1]{GJ99})\label{lem1}
Let $c$ be a local maximum of $f$. Then $c$ satisfies:
\begin{enumerate}
\item  $c$ lies in the interior of $\mathcal{C}$.
\item Either $a_i= \frac 13$ for all $i=1,2,3$, or there are $0< \alpha < \frac 1 \beta < \alpha' <1$, such that $a_i \in \{\alpha, \alpha'\}$ for all $i=1,2,3$. In the latter case there is exactly one $a_i$ equal to $\alpha'$, while all the other $a_j, j \neq i$ are equal to $\alpha$.
\end{enumerate}
\end{lemma}
Analyzing the function $f$ around the point $(\frac 13, \frac 13, \frac 13)$ we find that (in accordance with Lemma \ref{lem1}) it might be a local but not a global maximum of $f$, if $\beta>\beta_c= 4 \log 2$ is not too large (a similar observation was already made in \cite{Kestenschonman}):
\begin{lemma}\label{lemma loc min}
If $4 \log2 = \beta_c \le \beta <3$ then $(\frac{1}{3}, \frac{1}{3}, \frac{1}{3})$ is a local mode of $\pi_\beta$, if $N$ is large enough, i.e. $(\frac{1}{3}, \frac{1}{3}, \frac{1}{3})$ is a local maximum point of $\pi_\beta$.
\end{lemma}

\begin{proof}
In view of \eqref{stirling} is suffices to analyze $f$. For $x>0$, and $a \in [0,1]$ consider
$h(x):=f(\frac{1}{3}+x, \frac{1}{3}-ax, \frac{1}{3}-(1-a)x))$. It is easy matter to check that $h'(0)=0$ and $h''(0)= -(6-2\beta)(a^2+a+1)$. The assertion follows.
\end{proof}

\begin{remark}
Lemma \ref{lemma loc min} is a main reason why Theorem \ref{Centraltheorem} is true. It is not difficult to check that the same behavior is true for general $q \ge 3$ in an appropriately chosen temperature regime (depending on $q$). Therefore, also Theorem \ref{Centraltheorem} could be proven for general $q \ge 3$. Indeed, the property shown in Lemma \ref{lemma loc min} is intrinsically related to the first order phase transition of the mean-field Potts model. Such a phase transition can be characterized by the discontinuous transition of the accumulation point(s) of an order parameter of the model at the critical inverse temperature $\beta_c$. In the Potts model this order parameter is the variable $m_N$. However, in most natural models, these new accumulation point(s) are already local maxima of the distribution of the order parameter for some smaller values of $\beta$. Similarly, the old accumulation point(s) remain local maxima of the distribution of the order parameter for some larger values of $\beta$. This is exactly the statement of Lemma \ref{lemma loc min}.
\end{remark}

\medskip
Another key ingredient of the proof is a conductance argument (also known as Cheeger's inequality in \cite{DiaconisStrook_GeometricBoundsMC})
\begin{theorem}(\cite{JerrumSinclair})\label{conductance}
Let $P$ be a Markov chain
on a finite set $S$. Assume it is reversible with respect to $\pi$. For all $S' \subseteq S$, define
$$
\Phi_{S'} =\frac{\sum_{x \in S', y\notin S' }\pi(x)P(x,y)}{\pi(S')}.
$$
The conductance $\Phi$ given by
$$
\Phi = \min_{S':\pi(S')\le \frac 12} \Phi_{S'}.
$$
Then the following holds true for the spectral gap $\Gamma(P)$ of $P$:
$$ \frac{\Phi^2}2 \le \Gamma(P) \le  2\Phi.$$
\end{theorem}
As follows e.g.\ from \cite{DiaconisStrook_GeometricBoundsMC} a spectral gap that is the inverse of a polynomial in the problem instance results in fast mixing of the Markov chain. On the other hand, if the spectral gap is the inverse of an exponential in the problem instance, the Markov chain mixes slowly.
An immediate consequence of Theorem \ref{conductance} is that the Metropolis algorithm alone is slowly mixing on the low temperature Potts model.
\begin{proposition}\label{slowmixPotts}
The Metropolis algorithm mixes slowly on the Potts model, if $\beta >\beta_c$.
\end{proposition}
\begin{proof}
Take the macro-state $\mathfrak{a}_1:=\mathfrak{a}_1(\beta)$, i.e. the maximum point $a=(a_1,a_2,a_3)$ of $f$, where $a_1>a_2=a_3$. This point exists according to Lemma \ref{lem1} and because we are in the low temperature region. Since $\mathfrak{a}_1$
is a maximum of $f$, there is $\vep>0$ such that $f$ is decreasing on the ball of radius $2\vep$ centered in $\mathfrak{a}_1$, $B_{2\vep}(\mathfrak{a}_1)$, when we walk from the center to the boundary on a straight line. $\mathfrak{a}_1$ is one of the three points in which the distribution of $m_N$ concentrates for large $N$ and that are equally likely. Thus for
$$\mathfrak{B}_{2\vep}:=\{\sigma: m_N(\sigma)\in B_{2\vep}(\mathfrak{a}_1)\}
$$
we obtain that $$\frac 14 \le \pi_\beta(m_N(\sigma)\in \mathfrak{B}_{2\vep})\le \frac 13,$$ when $N$ is large enough and $\vep>0$ is fixed and small enough. Moreover, due to the exponential structure of $\pi_\beta$, i.e.
$$
\pi_\beta(m_N(\sigma)=c)=\frac{\exp(N(f(c)+\Delta(c)))}{N Z_\beta},
$$
and the behavior of $f$ on $B_{2\vep}(\mathfrak{a}_1)$ (on $B_{2\vep}(\mathfrak{a}_1)$, the function $f$ decreases like a multiple of the square of the two norms)
we obtain that
$$\frac{\pi_\beta(m_N(\sigma) \in \mathfrak{B}_{2\vep}\setminus \mathfrak{B}_{\vep})}{\pi_\beta(m_N(\sigma) \in \mathfrak{B}_{2\vep})}\le e^{-c'N}$$
for a suitably chosen constant $c'>0$. But this implies that the set $\mathfrak{B}_{2\vep}$ constitutes a ''bad cut''. Indeed with the notation of the previous theorem we see that
\begin{eqnarray*}
\Phi \le \Phi_{\mathfrak{B}_{2\vep}}&=& \frac{\sum_{x \in \mathfrak{B}_{2\vep}, y\notin \mathfrak{B}_{2\vep} }\pi_\beta(x)T_\beta(x,y)}{\pi_\beta(\mathfrak{B}_{2\vep})}\\
&\le& \frac{\sum_{x \in \mathfrak{B}_{2\vep}\setminus \mathfrak{B}_{\vep}, y\notin \mathfrak{B}_{2\vep} }\pi_\beta(x)T_\beta(x,y)}{\pi_\beta(B_{2\vep})}\\
&\le& \frac{\pi_\beta(m_N(\sigma)\in \mathfrak{B}_{2\vep}\setminus \mathfrak{B}_{\vep})}{\pi_\beta(m_N(\sigma)\in\mathfrak{B}_{2\vep})}\le e^{-c'N}
\end{eqnarray*}
Thus $T_\beta$ mixes slowly, when $\beta >\beta_c$.
\end{proof}
As a consequence, if EES is fast on the low temperature Potts model, this will have to be caused by the equi-energy steps. However, the following important observation is that we will not be able to switch between two states that are at very different distances from the center mode $\mathfrak{a}_0:=(\frac 13, \frac 13, \frac 13)$ by an equi-energy step. More precisely:
\begin{lemma}\label{lemsigmatau}
For each $\vep> 0$ and each $\vep>\delta > 0$ there is a number of spins $N_0$
such that for all $N > N_0$ and whenever $\sigma$ and $\tau$ satisfy
$||m_N(\sigma)- \mathfrak{a}_0||_1< \delta$ and $||m_N(\tau)- \mathfrak{a}_0||_1 > \vep$ (where $||\cdot||_1$ denotes the 1-norm on $\mathcal{C}$)
then
$$
Q_M(\sigma, \tau ) = 0.
$$
Here $Q_M$ is defined in \eqref{Qi}.
\end{lemma}
\begin{proof}
The proof mainly shows that under the given conditions the energies of $\sigma$ and $\tau$ are too far apart from each other. Indeed, observe that $Q_M(\sigma, \tau ) >0$ requires $\sigma$ and $\tau$ to be in the same energy band. Thus there is $i\in \{0, \ldots M-1\}$ such that $h_i < H_N(\sigma), H_N(\tau) \le h_{i+1}$.
Now each $\sigma_{\mathfrak{a}_0}$ with $m_N(\sigma_{\mathfrak{a}_0}) =
{\mathfrak{a}_0}$ minimizes the energy to $H_N(\sigma_{\mathfrak{a}_0})=\frac N2 \times 3 \times \frac 19= \frac N6$. On the other hand, the states where all spins point into the same direction have maximal energy $\frac N 2$
cf. \eqref{energy}.

Thus, recalling that $M=dN$, the width of the energy bands is
$$
h_{i+1}-h_i = \frac 1 M (h_M - h_0)= \frac{\frac N2-\frac N6}M= \frac {1}{3d}.$$
Therefore, $\sigma$ and $\tau$ are only in the same energy band, if
$$
|H_N(\sigma)-H_N(\tau)|= \frac N2 \left(||m_N(\sigma)||_2^2-||m_N(\tau)||_2^2\right) \le \frac{1}{3d},
$$
i.e. if the two norms $||m_N(\sigma)||_2$ and $||m_N(\tau)||_2$ satisfy
$$
\left|\, ||m_N(\sigma)||_2^2-||m_N(\tau)||_2^2 \,\right| \le \frac{1}{6dN}.
$$
Since 1-norm and 2-norm are equivalent on $\mathcal{C}$ this proves the assertion.
%
%
\end{proof}
We will again use a conductance argument to prove Theorem \ref{Centraltheorem}. In order to prepare it
let us lift the balls $\mathfrak{B}_\vep$ to $\Omega^{M+1}$:
For $\vep>0$ let
$$
\mathfrak{B}^{M+1}_\vep:=\{ x\in \O^{M+1}: m_N(x_{M}) \in B_\vep(\mathfrak{a}_0)\}.
$$
From now on we will assume that $\beta_c < \beta <3$. Recall that then still $\mathfrak{a}_0$ is a local (but not a global) maximum of the function $f$. Let us fix $\vep >0$ so small, that still $f$ is decreasing on
$B_\vep(\mathfrak{a}_0)$ when we move away from the center (in particular, $\mathfrak{a}_0$ is the only mode of $\pi_\beta$ on $B_\vep(\mathfrak{a}_0)$). Moreover, let us fix $\delta < \vep$ and $N_0$ so large that even with two equi-energy steps and a Metropolis step in between, a $\sigma$ with $||m_N(\sigma)- \mathfrak{a}_0||_1 > \vep$ cannot be reached from a $\tau$ with $||m_N(\tau)- \mathfrak{a}_0||_1 < \delta$.

This can be constructed as in Lemma \ref{lemsigmatau}. Indeed, we will need the following: For $\delta'>0$ given with $\delta<\delta'<\vep$ there is $N_1$, such that if $N \ge N_1$ an equi-energy jump started in $m_N\in B_\delta(\mathfrak{a}_0)$ will not leave $B_{\delta'}(\mathfrak{a}_0)$. The subsequent Metropolis step can only increase the 1-distance of $m_N$ to $\mathfrak{a}_0$ by at most $1/N$, hence $m_N$ is still in, say, $B_{\delta''}(\mathfrak{a}_0)$, for some $\delta'< \delta''<\vep$. Finally, there is $N_2$, such that if $N \ge N_2$ an equi-energy jump started in $m_N\in B_{\delta''}(\mathfrak{a}_0)$ will not leave $B_{\vep}(\mathfrak{a}_0)$. We will from now on always take $N \ge N_0 := \max\{N_1, N_2\}$.

All this is necessary because the chains $\mathcal{R}$ and $\mathcal{S}$ possibly comprise two such jumps.
Next we prove
\begin{lemma}
Let $\beta_c < \beta <3$ and $\vep>\delta>0$ and $N_0$ be chosen as above. Then, there exists $c''>0$ such that for $\tilde \pi(x):=\prod_{i=0}^M \pi_{\beta_i}(x_i)$, $x\in \O^{M+1}$, we have
$$
\frac{\pi(\mathfrak{B}^{M+1}_\vep\setminus\mathfrak{B}^{M+1}_\delta)}{\pi(\mathfrak{B}^{M+1}_\vep)}\le e^{-c'' N}.
$$
\end{lemma}
\begin{proof}
According to our above analysis $\mathfrak{a}_0$ is a local (but not a global) maximum point of the distribution of $m_N$ under $\pi_\beta=\pi_{\beta_M}=:\pi_M$, if $\beta_c < \beta <3$. Therefore
$$
\frac{\pi_M(\{\sigma: m_N(\sigma) \in \mathfrak{B}_\vep(\mathfrak{a}_0) \setminus\mathfrak{B}_\delta(\mathfrak{a}_0)\})}{\pi(\{\sigma: m_N(\sigma) \in \mathfrak{B}_\vep (\mathfrak{a}_1)\})}\le e^{-c'' N}
$$
for $c''>0$ chosen appropriately. The proof of this statement follows the concepts of the proof of Proposition \ref{slowmixPotts}.
This fact is easily transferred to the measure $\tilde \pi$ due to its product structure.
\end{proof}
With the help of this lemma we will be able to establish that the set $\mathfrak{B}^{M+1}_\vep$ constitutes a ''bad cut'' for the Markov chain $\mathcal{S}$.
\begin{proposition}\label{condest}
Consider the Markov chain $\mathcal{S}$ on the state space $\Omega_{EES}:= \Omega^{M+1} \times \mathfrak{M}$ where again $\mathfrak{M}$ is the space of all $(M+1) \times 3^N$ matrices. Here the first coordinate keeps record of the current state of $M+1$ chains, while the second coordinate tracks the filling of the matrix $\mathcal{M}$.

If $\beta_c < \beta <3$ and the second coordinate of $\mathcal{S}$ is equal to $M_0$ its conductance $\Phi(\mathcal{S})$ satisfies
$$
\Phi(\mathcal{S})  \le e^{-c'' N}
$$
for some $c''>0$, if $N$ is large enough.
\end{proposition}
\begin{proof}
Since $\beta_c < \beta$ clearly $\pi_M(\{\sigma: m_N(x_{M}) \in B_\vep(\mathfrak{a}_0)\})< \frac 12$, when $N$ is large enough. Thus also
$$\pi(\Omega^M \times  \{\sigma: m_N(x_{M}) \in B_\vep(\mathfrak{a}_0)\} \times M_0)=
\pi (\mathfrak{B}_\vep^{M+1} \times M_0) < \frac 12
$$
for $N$ large enough.
Thus
\begin{eqnarray*}
\Phi(\mathcal{S}) &\le& \Phi_{\mathfrak{B}_\vep^{M+1} \times M_0}(\mathcal{S})= \frac{\sum_{{\tilde \sigma \in \mathfrak{B}_\vep^{M+1} \times M_0}\atop{\tilde \tau \notin \mathfrak{B}_\vep^{M+1} \times M_0}}\pi(\tilde{\sigma})\mathcal{S}(\tilde{\sigma}, \tilde{\tau})}{\pi(\mathfrak{B}_\vep^{M+1} \times M_0)}\\
&\le& \frac{\sum_{{\tilde \sigma \in (\mathfrak{B}_\vep^{M+1} \times M_0)\setminus
(\mathfrak{B}_\delta^{M+1} \times M_0) }\atop{\tilde \tau \notin \mathfrak{B}_\vep^{M+1} \times M_0}}\pi(\tilde{\sigma})\mathcal{S}(\tilde{\sigma}, \tilde{\tau})}{\pi(\mathfrak{B}_\vep^{M+1} \times M_0)}\\
&\le& \frac{\sum_{{\tilde \sigma \in (\mathfrak{B}_\vep^{M+1} \times M_0)\setminus
(\mathfrak{B}_\delta^{M+1} \times M_0) }\atop{\tilde \tau \notin \mathfrak{B}_\vep^{M+1} \times M_0}}\pi(\tilde{\sigma})}{\pi(\mathfrak{B}_\vep^{M+1} \times M_0)}\\
&=& \frac{\pi\left((\mathfrak{B}^{M+1}_\vep \times M_0)\setminus(\mathfrak{B}^{M+1}_\delta\times M_0)\right)}{\pi(\mathfrak{B}^{M+1}_\vep\times M_0)}\\
&=& \frac{\tilde \pi(\mathfrak{B}^{M+1}_\vep\setminus\mathfrak{B}^{M+1}_\delta)}{\tilde \pi(\mathfrak{B}^{M+1}_\vep)}\le e^{-c'' N}.
\end{eqnarray*}
Here we used, of course, our previous estimates together with our construction of $\delta$. Starting in $B_\delta(\mathfrak{a}_0)$ the combination of an equi-energy jump, a Metropolis move, and another equi-energy jump will not leave $B_{\vep}(\mathfrak{a}_0)$ according to Lemma \ref{lemsigmatau} and the construction of $\delta$ and $N_0$.
\end{proof}
Now we have prepared everything to prove Theorem \ref{Centraltheorem}.
\begin{proof}[Proof of Theorem \ref{Centraltheorem}]
Just note, that conditioned on the event that the second coordinate of $\mathcal{S}$ is in $M_0$ (which it cannot leave anymore), $\mathcal{S}$ is reversible with respect to $\pi$. Hence we can apply Theorem \ref{conductance} together with the conductance estimate of Proposition \ref{condest} to obtain the desired result.
\end{proof}
\begin{remark}
Note that a similar proof would not work in the Curie-Weiss model, because there the ''center point'' $(1/2, 1/2)$, i.e. the $\sigma$'s where both directions for the spins occur equally often, is always a local minimum of the Gibbs measure at low temperatures.

Moreover, note that we could adapt the proof to different values of $q\ge 3$ as mentioned above.

Finally, a similar argument should work for ''more disordered'' models, as Potts models on sufficiently dense Erd\"os-R\'enyi graphs, as e.g. analyzed in \cite{KLS19a} for $q=2$.
\end{remark}

\begin{remark}
We have just seen that EES mixes slowly on the 3-state Potts model at $\beta_c<\beta<3$, {\it even when
we know the energies of the entire set of states}. We also argued that at high temperatures these temperature steps are not necessary, because already the Metropolis chain itself converges rapidly. However, one may doubt that there are reasonable models, at all, in which EES converges rapidly while the Metropolis algorithm does not. The point is, that, if we have not filled $\mathcal{M}$ almost entirely, a temperature jump may provide the desired tunneling effect, but to a rather unfavorable point of the target distribution.
\end{remark}


\begin{thebibliography}{}

\bibitem[Andrieu et~al., 2008]{AJDD}
Andrieu, C., Jasra, A., Doucet, A., and Del~Moral, P. (2008).
\newblock A note on convergence of the equi-energy sampler.
\newblock {\em Stoch. Anal. Appl.}, 26(2):298--312.

\bibitem[Baragatti et~al., 2013]{Baragatti_etal}
Baragatti, M., Grimaud, A., and Pommeret, D. (2013).
\newblock Parallel tempering with equi-energy moves.
\newblock {\em Stat. Comput.}, 23(3):323--339.

\bibitem[Berg, 2000]{B00}
Berg, B.~A. (2000).
\newblock Introduction to multicanonical {M}onte {C}arlo simulations.
\newblock In {\em Monte {C}arlo methods ({T}oronto, {ON}, 1998)}, volume~26 of
  {\em Fields Inst. Commun.}, pages 1--24. Amer. Math. Soc., Providence, RI.

\bibitem[Berg and Neuhaus, 1992]{BN92}
Berg, B.~A. and Neuhaus, T. (1992).
\newblock Multicanonical ensemble: A new approach to simulate first-order phase
  transitions.
\newblock {\em Phys. Rev. Lett.}, 68:9--12.

\bibitem[Bhatnagar and Randall, 2004]{BhatnagarRandallTorpidMixingPotts}
Bhatnagar, N. and Randall, D. (2004).
\newblock Torpid mixing of simulated tempering on the {P}otts model.
\newblock In {\em Proceedings of the {F}ifteenth {A}nnual {ACM}-{SIAM}
  {S}ymposium on {D}iscrete {A}lgorithms}, pages 478--487 (electronic), New
  York. ACM.

\bibitem[Bhatnagar and Randall, 2016]{BR16}
Bhatnagar, N. and Randall, D. (2016).
\newblock Simulated tempering and swapping on mean-field models.
\newblock {\em J. Stat. Phys.}, 164(3):495--530.

\bibitem[Cuff et~al., 2012]{Cuffetal}
Cuff, P., Ding, J., Louidor, O., Lubetzky, E., Peres, Y., and Sly, A. (2012).
\newblock Glauber dynamics for the mean-field {P}otts model.
\newblock {\em J. Stat. Phys.}, 149(3):432--477.

\bibitem[Diaconis and Stroock, 1991]{DiaconisStrook_GeometricBoundsMC}
Diaconis, P. and Stroock, D. (1991).
\newblock {Geometric bounds for eigenvalues of Markov chains.}
\newblock {\em Ann. Appl. Probab.}, 1(1):36--61.

\bibitem[Doll et~al., 2018]{DDN18}
Doll, J., Dupuis, P., and Nyquist, P. (2018).
\newblock A large deviations analysis of certain qualitative properties of
  parallel tempering and infinite swapping algorithms.
\newblock {\em Appl. Math. Optim.}, 78(1):103--144.

\bibitem[Ebbers et~al., 2014]{EKLV}
Ebbers, M., Kn\"{o}pfel, H., L\"{o}we, M., and Vermet, F. (2014).
\newblock Mixing times for the swapping algorithm on the
  {B}lume-{E}mery-{G}riffiths model.
\newblock {\em Random Structures Algorithms}, 45(1):38--77.

\bibitem[Ebbers and L{\"o}we, 2009]{EbbersLoweREM}
Ebbers, M. and L{\"o}we, M. (2009).
\newblock {Torpid mixing of the swapping chain on some simple spin glass
  models.}
\newblock {\em Markov Process. Relat. Fields}, 15(1):59--80.

\bibitem[Ellis and Wang, 1990]{EllisWang}
Ellis, R.~S. and Wang, K. (1990).
\newblock Limit theorems for the empirical vector of the
  {C}urie-{W}eiss-{P}otts model.
\newblock {\em Stochastic Process. Appl.}, 35(1):59--79.

\bibitem[Geyer, 1991]{GeyerMCMCmaximumLikelihood}
Geyer, C.~J. (1991).
\newblock Markov chain monte carlo maximum likelihood.
\newblock In {\em Computing Science and Statistics: Proceedings of 23rd
  Symposium on the Interface Interface Foundation}, pages 156--163. Fairfax
  Station.

\bibitem[Geyer and Thompson, 1995]{GeyerThompsonAMCMC}
Geyer, C.~J. and Thompson, E.~A. (1995).
\newblock {Annealing Markov chain Monte Carlo with applications to ancestral
  inference.}
\newblock {\em J. Am. Stat. Assoc.}, 90(431):909--920.

\bibitem[Gore and Jerrum, 1999]{GJ99}
Gore, V.~K. and Jerrum, M.~R. (1999).
\newblock The {S}wendsen-{W}ang process does not always mix rapidly.
\newblock {\em J. Statist. Phys.}, 97(1-2):67--86.

\bibitem[H\"{a}ggstr\"{o}m, 2002]{haeggstroem}
H\"{a}ggstr\"{o}m, O. (2002).
\newblock {\em Finite {M}arkov chains and algorithmic applications}, volume~52
  of {\em London Mathematical Society Student Texts}.
\newblock Cambridge University Press, Cambridge.

\bibitem[Hua and Kou, 2011]{HuaKou}
Hua, X. and Kou, S.~C. (2011).
\newblock Convergence of the equi-energy sampler and its application to the
  {I}sing model.
\newblock {\em Statist. Sinica}, 21(4):1687--1711.

\bibitem[Hukushima and Nemoto, 1996]{exMC}
Hukushima, K. and Nemoto, K. (1996).
\newblock Exchange monte carlo method and application to spin glass
  simulations.
\newblock {\em Journal of the Physical Society of Japan}, 65(6):1604--1608.

\bibitem[Kabluchko et~al., 2019]{KLS19a}
Kabluchko, Z., L\"{o}we, M., and Schubert, K. (2019).
\newblock Fluctuations of the magnetization for {I}sing models on dense
  {E}rd\''os-{R}\'{e}nyi random graphs.
\newblock {\em J. Stat. Phys.}, 177(1):78--94.

\bibitem[{Kesten} and {Schonmann}, 1989]{Kestenschonman}
{Kesten}, H. and {Schonmann}, R.~H. (1989).
\newblock {Behavior in large dimensions of the Potts and Heisenberg models.}
\newblock {\em {Rev. Math. Phys.}}, 1(2-3):147--182.

\bibitem[Kou et~al., 2006]{KZW}
Kou, S.~C., Zhou, Q., and Wong, W.~H. (2006).
\newblock Equi-energy sampler with applications in statistical inference and
  statistical mechanics.
\newblock {\em Ann. Statist.}, 34(4):1581--1652.
\newblock With discussions and a rejoinder by the authors.

\bibitem[L\"{o}we and Meise, 2001]{LM01}
L\"{o}we, M. and Meise, C. (2001).
\newblock Note on the knapsack {M}arkov chain.
\newblock {\em Stochastic Process. Appl.}, 94(1):155--170.

\bibitem[L{\"o}we and Vermet, 2009]{loewe_vermet_swap}
L{\"o}we, M. and Vermet, F. (2009).
\newblock { The swapping algorithm for the Hopfield model with two patterns.}
\newblock {\em Stochastic Process. Appl.}, 119(10):3471--3493.

\bibitem[Madras, 1998]{madras}
Madras, N. (1998).
\newblock Umbrella sampling and simulated tempering.
\newblock In {\em Numerical methods for polymeric systems, Ed., S. G.
  Whittington, IMA Volume in Mathematics and Its Applications 102}, pages
  19--32, New York. Springer-Verlag.

\bibitem[Madras and Piccioni, 1999]{MadrasPiccioni}
Madras, N. and Piccioni, M. (1999).
\newblock Importance sampling for families of distributions.
\newblock {\em Ann. Appl. Prob.}, 9(4):1202--1225.

\bibitem[Madras and Zheng, 2003]{MadrasZhengCW}
Madras, N. and Zheng, Z. (2003).
\newblock {On the swapping algorithm.}
\newblock {\em Random Struct. Algorithms}, 22(1):66--97.

\bibitem[Marinari and Parisi, 1992]{marinari_parisi}
Marinari, E. and Parisi, G. (1992).
\newblock {Simulated tempering: A new Monte Carlo scheme.}
\newblock {\em Europhys Lett.}, 19(6):451--458.

\bibitem[Mossel and Sly, 2013]{MosselSly}
Mossel, E. and Sly, A. (2013).
\newblock Exact thresholds for {I}sing-{G}ibbs samplers on general graphs.
\newblock {\em Ann. Probab.}, 41(1):294--328.

\bibitem[Orlandini, 1998]{orlandini}
Orlandini, E. (1998).
\newblock Monte carlo study of polymer systems by multiple markov chain method.
\newblock In {\em Numerical methods for polymeric systems, Ed., S. G.
  Whittington, IMA Volume in Mathematics and Its Applications 102}, pages
  33--57, New York. Springer-Verlag.

\bibitem[Sinclair and Jerrum, 1989]{JerrumSinclair}
Sinclair, A. and Jerrum, M. (1989).
\newblock Approximate counting, uniform generation and rapidly mixing {M}arkov
  chains.
\newblock {\em Inform. and Comput.}, 82(1):93--133.

\bibitem[Torrie and Valleau, 1977]{umbrella}
Torrie, G. and Valleau, J. (1977).
\newblock Nonphysical sampling distributions in monte carlo free-energy
  estimation: Umbrella sampling.
\newblock {\em Journal of Computational Physics}, 23(2):187 -- 199.

\end{thebibliography}
\end{document}